\documentclass{svjour3}

\smartqed

\usepackage{amsmath}
\usepackage{amssymb}
\usepackage{amscd}
\usepackage{mathtools}

\usepackage{graphicx}
\usepackage{epsfig}
\usepackage{epstopdf}
\usepackage[caption=false]{subfig}

\usepackage[utf8]{inputenc}
\usepackage[T1]{fontenc}

\providecommand{\R}{\mathbb{R}}

\providecommand{\set}[1]{\left\{#1\right\}}

\begin{document}

\title{Improved conditioning of the Floater--Hormann interpolants
	\thanks{The author is grateful for the support of the National Science Foundation under Grant No. DMS 1547357.}
}

\author{Jeremy K. Mason}

\institute{J.K. Mason \at
	The Ohio State University, Columbus, OH 43206, USA.\\
	\email{jkylemason@gmail.com}           
}

\date{Received: date / Accepted: date}

\maketitle

\begin{abstract}
The Floater--Hormann family of rational interpolants do not have spurious poles or unattainable points, are efficient to calculate, and have arbitrarily high approximation orders. One concern when using them is that the amplification of rounding errors increases with approximation order, and can make balancing the interpolation error and rounding error difficult. This article proposes to modify the Floater--Hormann interpolants by including additional local polynomial interpolants at the ends of the interval. This appears to improve the conditioning of the interpolants and allow higher approximation orders to be used in practice.

\end{abstract}

\section{Introduction}
\label{sec:introduction}

Let $f: [a, b] \rightarrow \R$ be an unknown function to be interpolated, $x_0, x_1 \dots x_n$ be strictly increasing values in the interval $[a, b]$, and $y_0, y_1 \dots y_n$ be measured function values at these points. Depending on the differentiability of $f$ and the distribution of the $x_i$, any one of a number of interpolants in the literature could be used. For example, if the distribution of the $x_i$ is not fixed, then the interpolating polynomial of minimum degree is accurate and stable when the $x_i$ are chosen as one of the various kinds of Chebyshev points \cite{2013trefethen}. If some of the $y_i$ are outliers and a high degree of differentiability is not required, then spline interpolation is efficient and confines the effect of the outliers to short subintervals \cite{2001deboor}.

Rational interpolants (formed by the ratio of two polynomials) offer an intriguing alternative to polynomials and piecewise polynomials like the ones above. Since the space of rational functions contains that of polynomials, rational interpolants can accurately approximate more diverse function behaviors. The theory of rational interpolants is less developed than that of polynomials though, and some of the available constructions suffer from the appearance of spurious poles on the real line and unattainable points. This is inconvenient enough to have contributed to the historically limited use of rational interpolation in practice.

Among the rational interpolants, the Floater--Hormann family \cite{2007floater} is notable for a provable absence of unattainable points and poles along the real line, high rates of approximation, and a simple construction. Let $p_{i, j}(x)$ be the unique polynomial of minimum degree that passes through the points $(x_i, y_i) \dots (x_j, y_j)$ for $i  < j$, and $\chi_{i, j}(x)$ be defined as
\begin{equation}
\chi_{i, j}(x) = (-1)^{i} \prod_{k = i}^{j} \frac{1}{x - x_k}.
\label{eq:fh_chi}
\end{equation}
Then for any integer $0 \leq d \leq n$, the Floater--Hormann interpolant of degree $d$ is a blend of polynomial interpolants through successive sets of $d + 1$ points:
\begin{equation}
r^{(d)}(x) = \frac{\sum_{i = 0}^{n - d} \chi_{i, i + d}(x) p_{i, i + d}(x)}{\sum_{i = 0}^{n - d} \chi_{i, i + d}(x)}.
\label{eq:fh_blend}
\end{equation}
This form is not often used for computations though, since the barycentric form is more computationally efficient and numerically stable \cite{2004berrut}. The present derivation of the barycentric form closely follows that of Floater and Hormann \cite{2007floater}, but introduces some notation that will be useful in subsequent sections.

The derivation of the barycentric form begins by writing the polynomials $p_{i, j}(x)$ in the Lagrange form:
\begin{equation}
p_{i, j}(x) = \sum_{k = i}^{j} \prod_{\substack{l = i \\ l \neq k}}^{j} \frac{x - x_l}{x_k - x_l} y_k.
\label{eq:lagrange}
\end{equation}
Let $t^{(d)}(x)$ be the numerator of $r^{(d)}(x)$ in Equation \ref{eq:fh_blend}. Substituting the definitions of $\chi_{i, j}(x)$ and $p_{i, j}(x)$ from Equations \ref{eq:fh_chi} and \ref{eq:lagrange} and canceling common factors gives
\begin{equation}
t^{(d)}(x) = \sum_{i = 0}^{n - d} \sum_{j = i}^{i + d} \frac{(-1)^{i}}{x - x_j} \prod_{\substack{k = i \\ k \neq j}}^{i + d} \frac{1}{x_j - x_k} y_j.
\label{eq:fh_numerator}
\end{equation}
It will be convenient in the following to introduce a symbol for the barycentric weights, i.e., the constants in the inner summation:
\begin{equation*}
\omega_{i, j, k} = \prod_{\substack{l = i \\ l \neq j}}^{k} \frac{1}{x_j - x_l}.
\end{equation*}
Exchanging the order of the summations in Equation \ref{eq:fh_numerator} and defining the Floater--Hormann weights as
\begin{equation*}
\xi_{j}^{(d)} = \sum_{i = \max(0, j - d)}^{\min(j, n - d)} (-1)^{i} \omega_{i, j, i + d}
\end{equation*}
allows the numerator of the interpolant to be written as
\begin{equation}
t^{(d)}(x) = \sum_{j = 0}^{n} \frac{\xi_{j}^{(d)}}{x - x_j} y_j.
\label{eq:fh_bary}
\end{equation}
The denominator of the interpolant can be found by requiring that a constant function be interpolated exactly. That is, the denominator is the same as the numerator when all of the $y_j$ are equal to one. This gives
\begin{align}
r^{(d)}(x) & = \sum_{j = 0}^{n} \bigg( \frac{\xi_{j}^{(d)}}{x - x_j} \bigg/ \sum_{k = 0}^{n} \frac{\xi_{k}^{(d)}}{x - x_k} \bigg) y_j 
\label{eq:fh_bary_full} \\
& = \sum_{j = 0}^{n} \beta^{(d)}_j(x) y_j
\label{eq:fh_basis}
\end{align}
for the barycentric form of the Floater--Hormann interpolant of degree $d$, where $\beta^{(d)}_j(x)$ is the $j$th basis function.

\begin{figure}
\center
\subfloat[]{%
	\label{subfig:fh_blending}{%
		\includegraphics[height=37mm]{%
			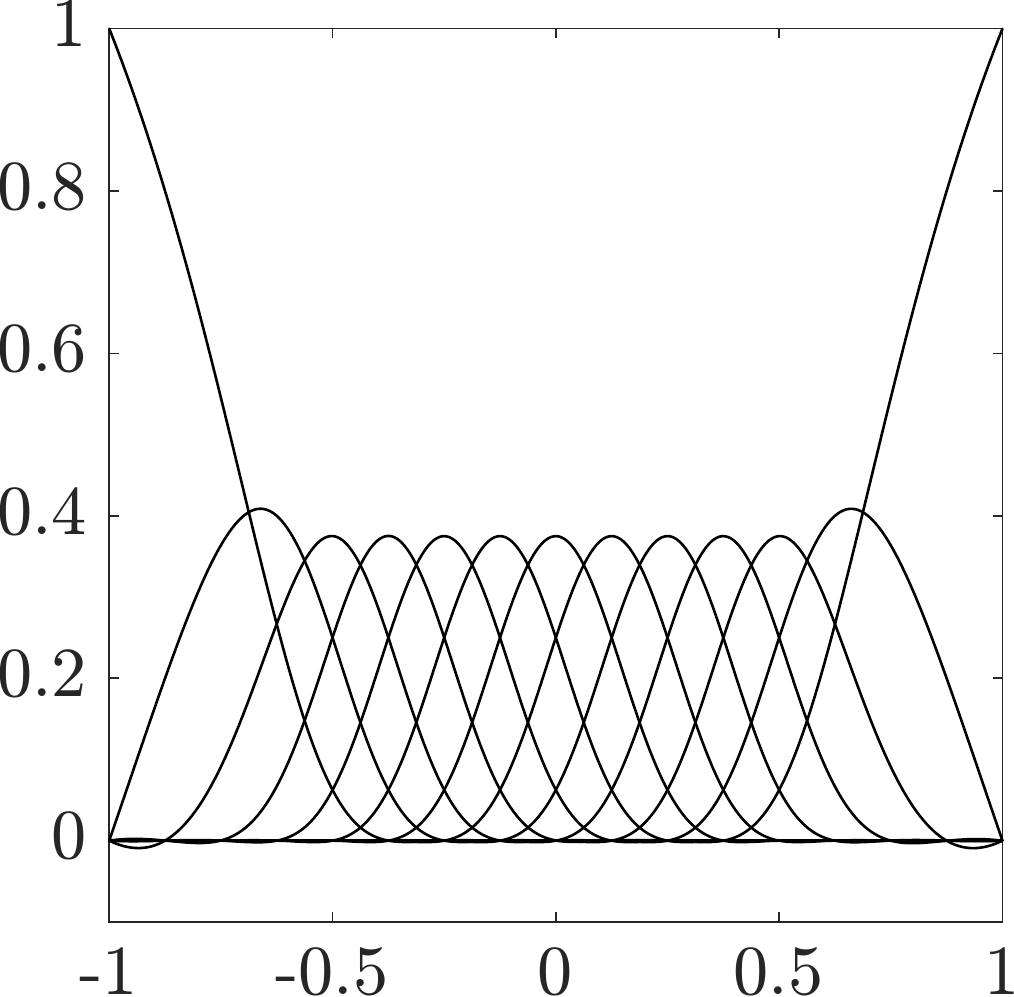}}} \quad
\subfloat[]{%
	\label{subfig:fh_basis}{%
		\includegraphics[height=37mm]{%
			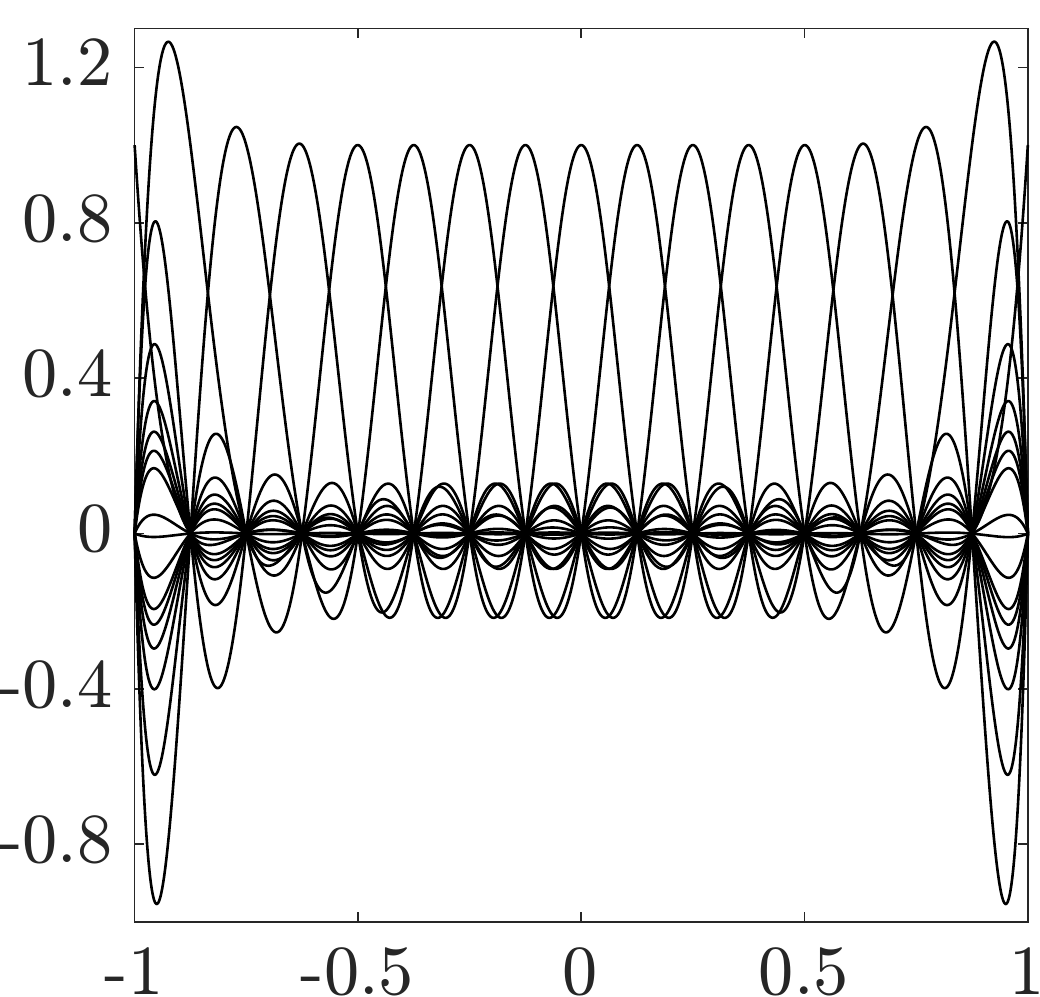}}} \quad
\subfloat[]{%
	\label{subfig:fh_lebesgue}{%
		\includegraphics[height=37mm]{%
			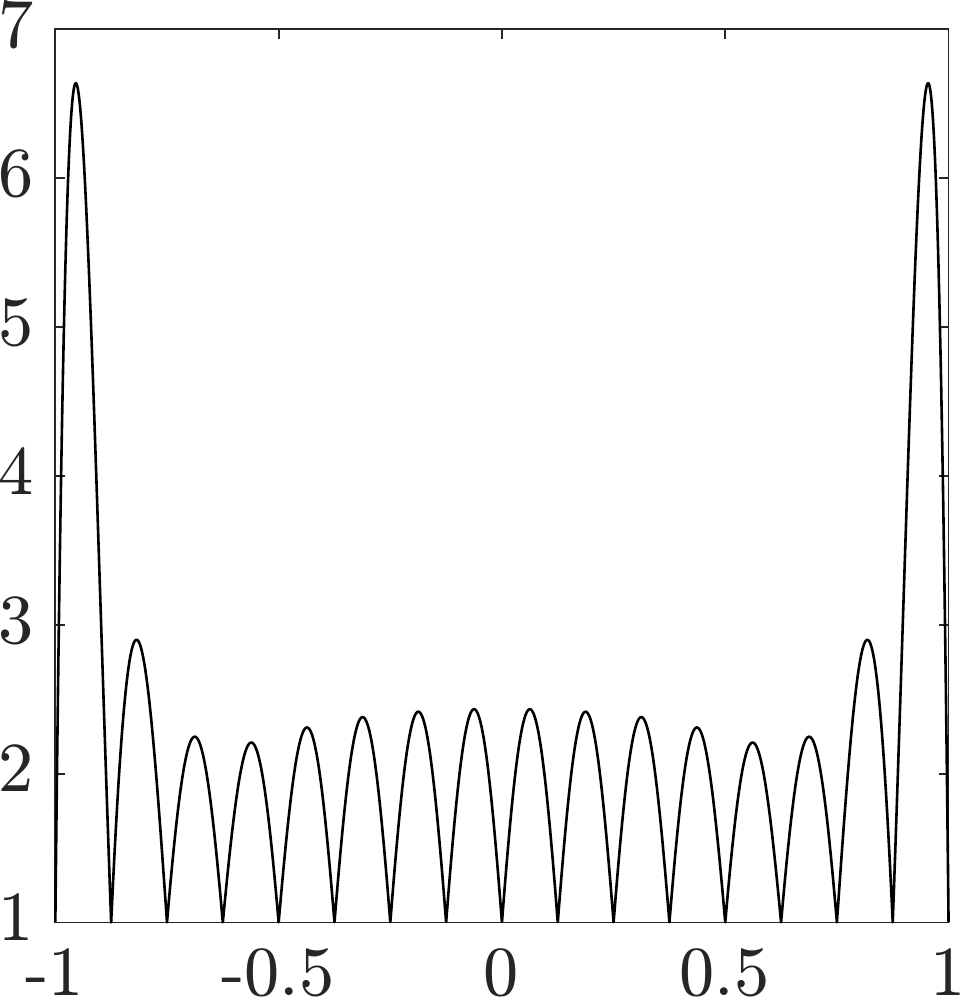}}}
\caption{\label{fig:fh_example}For the Floater--Hormann interpolant with $d = 4$ and $n = 16$ for equispaced nodes in the interval $[-1, 1]$, (a) the blending functions $\chi_{i, i + d}(x) \big/ \sum_{j = 0}^{n - d} \chi_{j, j + d}(x)$ in Equation \ref{eq:fh_blend}, (b) the basis functions $\beta^{(4)}_j(x)$ in Equation \ref{eq:fh_basis}, and (c) the Lebesgue function $\lambda^{(4)}(x)$.}
\end{figure}

The properties of the Floater--Hormann interpolants for equispaced $x_i$ are of particular interest since this situation arises often in practice and is difficult for polynomial interpolants. For specificity, consider the Floater--Hormann interpolant with $d = 4$ and $n = 16$ for equispaced nodes in the interval $[-1, 1]$. The basis functions $\beta^{(4)}_j(x)$ appear in Figure \ref{subfig:fh_basis}, and the Lebesgue function $\lambda^{(4)}(x) = \sum_{j} | \beta^{(4)}_j(x) |$ in Figure \ref{subfig:fh_lebesgue}. The Lebesgue function effectively indicates the relative condition number of the interpolant, i.e., the sensitivity to measurement or rounding errors. This reveals one aspect of the Floater--Hormann interpolants that could be improved---the conditioning degrades at the ends of the interval. More specifically, the supremum of the Lebesgue function (the Lebesgue constant) increases exponentially with $d$ \cite{2012bos}. Since the rate of approximation also increases with $d$, one of the main concerns when using Floater--Hormann interpolants is finding a value of $d$ that appropriately balances the interpolation error and the rounding error \cite{2012guttel}.

It is worthwhile to consider the source of this ill-conditioning. The definition of the Lebesgue function and Figure \ref{subfig:fh_basis} reveal that it is caused by alternating oscillations in the basis functions at the ends of the interval. The reason that this occurs only at the ends of the interval and not on the interior is most easily seen from Equation \ref{eq:fh_blend}, where the functions $\chi_{i, i + d}(x) \big/ \sum_{j = 0}^{n - d} \chi_{j, j + d}(x)$ blend the local polynomial interpolants. Figure \ref{subfig:fh_blending} shows that while the blending functions on the interior decay rapidly enough to damp the oscillations of the local polynomial interpolants, the blending functions at the ends of the interval do not decay at all. This suggests that the source of the ill-conditioning is a deficit of local interpolants at the ends of the interval.

One proposal to improve the conditioning of the Floater--Hormann interpolants extrapolates to points outside of the interval, and uses a Floater--Hormann interpolant on the extended set of points \cite{2013klein}. This does resolve the source of the ill-conditioning identified above, but at the cost of introducing instability by the extrapolation process \cite{2017decamargo}. There is evidence that even when the extrapolated points are computed in multiple precision arithmetic, the effect of measurement and rounding errors in the initial $y_i$ can negate any advantage of this approach \cite{2017decamargo}. Moreover, the extrapolation process obscures the dependence of the interpolants on the initial points, makes explicit basis functions difficult to construct, and complicates the study of the Lebesgue function. All of this means that some other procedure to improve the conditioning of the Floater--Hormann interpolants is greatly desired.

\section{An Alternative Interpolant}
\label{sec:alternative}

A possible approach to improve the conditioning of the Floater--Hormann interpolants would be to constrain the derivatives at the endpoints, effectively replacing several of the Lagrange interpolants in Equation \ref{eq:fh_blend} with Hermite interpolants. Since the original problem does not include any information about the derivatives of $f$, they would need to be estimated from, e.g., one-sided finite difference formulas. Such formulas are themselves based on polynomial interpolants though, and suffer from ill-conditioning when the nodes are equispaced and the degree of the polynomial is high. This could be mitigated by using finite difference formulas derived from polynomial interpolants of degree less than $d$, but the issue remains that any finite difference formulas would complicate the construction of explicit basis functions. Perhaps then the conditioning of the Floater--Hormann interpolants could be improved by directly including some polynomial interpolants of degree less than $d$, with suitably modified blending functions to confine their influence to the ends of the interval.

Let $\phi^{(d)}_{i}(x)$ and $\psi^{(d)}_{i}(x)$ be the modifications of $\chi_{i, j}(x)$ that will be used to blend polynomial interpolants through fewer than $d + 1$ points at the lower and upper ends of the interval:
\begin{align}
\phi^{(d)}_{i}(x) & = \frac{(-1)^{d - i}}{(x - x_0)^{d - i}} \chi_{0, i}(x)
\label{eq:phi} \\
\psi^{(d)}_{i}(x) & = \frac{1}{(x - x_n)^{i - n + d}} \chi_{i, n}(x).
\label{eq:psi}
\end{align}
Further define the three index sets $I_1 := \set{d - e \dots d - 1}$, $I_2 := \set{0 \dots n - d}$, and $I_3 := \set{n - d + 1 \dots n - d + e}$. Then for any integers $0 \leq d \leq n$ and $0 \leq e \leq d$, the proposed interpolant $r^{(d, e)}(x)$ is a Floater--Hormann interpolant of degree $d$ with $e$ additional polynomial interpolants through fewer points at the lower and upper ends of the interval:
\begin{equation}
r^{(d, e)}(x) = \frac{\sum_{i \in I_1} \phi^{(d)}_{i}(x) p_{0, i}(x) + \sum_{i \in I_2} \chi_{i, i + d}(x) p_{i, i + d}(x) + \sum_{i \in I_3} \psi^{(d)}_{i}(x) p_{i, n}(x)}{\sum_{i \in I_1} \phi^{(d)}_{i}(x) + \sum_{i \in I_2} \chi_{i, i + d}(x) + \sum_{i \in I_3} \psi^{(d)}_{i}(x)}.
\label{eq:alt_blend}
\end{equation}
As before, let $t^{(d, e)}(x)$ be the numerator of $r^{(d, e)}(x)$. Using Equation \ref{eq:fh_bary} for the numerator of the Floater--Hormann interpolant, substituting the definitions of $\phi^{(d)}_{i}(x)$, $\psi^{(d)}_{i}(x)$ and $p_{i, j}(x)$ from Equations \ref{eq:phi}, \ref{eq:psi} and \ref{eq:lagrange}, and canceling common factors gives
\begin{align}
t^{(d, e)}(x) = & \sum_{i = d - e}^{d - 1} \frac{(-1)^{d - i}}{(x - x_0)^{d - i}} \sum_{j = 0}^{i} \frac{\omega_{0, j, i}}{x - x_j} y_j + \sum_{j = 0}^{n} \frac{\xi_{j}^{(d)}}{x - x_j} y_j \nonumber \\
& + \sum_{i = n - d + 1}^{n - d + e} \frac{(-1)^{i}}{(x - x_n)^{i - n + d}} \sum_{j = i}^{n} \frac{\omega_{i, j, n}}{x - x_j} y_j
\label{eq:alt_numerator}
\end{align}
Analogous to the Floater--Hormann weights, define the pair of functions
\begin{align}
\zeta_{j}^{(d, e)}(x) & = \sum_{i = \max(j, d - e)}^{d - 1} \frac{(-1)^{d - i} \omega_{0, j, i}}{(x - x_0)^{d - i}} \nonumber \\
& = -\frac{\omega_{0, j, u}}{x - x_0} \bigg( 1 - \frac{x_j - x_{u}}{x - x_0} \bigg( \dots \bigg( 1 - \frac{x_j - x_{l + 1}}{x - x_0} \bigg) \bigg) \bigg)
\label{eq:zeta_horner}
\end{align}
\begin{align} 
\eta_{j}^{(d, e)}(x) & = \sum_{i = n - d + 1}^{\min(j, n - d + e)} \frac{(-1)^{i} \omega_{i, j, n}}{(x - x_n)^{i - n + d}} \nonumber \\
& = (-1)^{l} \frac{\omega_{l, j, n}}{x - x_0} \bigg( 1 - \frac{x_j - x_{l}}{x - x_0} \bigg( \dots \bigg( 1 - \frac{x_j - x_{u - 1}}{x - x_0} \bigg) \bigg) \bigg)
\label{eq:eta_horner}
\end{align}
where $l$ and $u$ are the lower and upper indices of the respective summations, and the second forms of $\zeta_{j}^{(d, e)}(x)$ and $\eta_{j}^{(d, e)}(x)$ follow from the application of Horner's method. Exchanging the order of summations over $i$ and $j$ in Equation \ref{eq:alt_numerator} allows $t^{(d, e)}(x)$ to be written as
\begin{align*}
t^{(d, e)}(x) & = \sum_{j = 0}^{d - 1} \frac{\zeta_{j}^{(d, e)}(x)}{x - x_j} y_j + \sum_{j = 0}^{n} \frac{\xi_{j}^{(d)}}{x - x_j} y_j + \sum_{j = n - d + 1}^{n} \frac{\eta_{j}^{(d, e)}(x)}{x - x_j} y_j \\
& = \sum_{j = 0}^{n} \frac{\zeta_{j}^{(d, e)}(x) + \xi_{j}^{(d)} + \eta_{j}^{(d, e)}(x)}{x - x_j} y_j
\end{align*}
where the second equality uses the convention that a sum is zero whenever the upper index is less than the lower index. Given this form for the numerator, the denominator of the proposed interpolant is again found by requiring that a constant function be interpolated exactly, i.e., by setting all of the $y_j$ equal to one. This gives
\begin{align}
r^{(d, e)}(x) & = \sum_{j = 0}^{n} \bigg( \frac{\zeta_{j}^{(d, e)}(x) + \xi_{j}^{(d)} + \eta_{j}^{(d, e)}(x)}{x - x_j} \bigg/ \sum_{k = 0}^{n} \frac{\zeta_{k}^{(d, e)}(x) + \xi_{k}^{(d)} + \eta_{k}^{(d, e)}(x)}{x - x_k} \bigg) y_j 
\label{eq:alt_bary_full} \\
& = \sum_{j = 0}^{n} \beta^{(d, e)}_j(x) y_j
\label{eq:alt_basis}
\end{align}
for the form of the proposed interpolant analogous to Equation \ref{eq:fh_bary_full}, where $\beta^{(d, e)}_j(x)$ is the $j$th basis function.

While Equation \ref{eq:alt_bary_full} mimics the barycentric form of the Floater--Hormann interpolants, multiplying the numerator and denominator of Equation \ref{eq:alt_blend} by
\begin{equation}
\ell(x) = (-1)^{n - d + e} (x - x_0)^{e + 1} (x - x_1) \dots (x - x_{n - 1}) (x - x_n)^{e + 1}
\label{eq:alt_all_factors}
\end{equation}
reveals that the proposed interpolant is a rational function with numerator and denominator degrees at most $n + 2e$ and $n - d + 2e$. This means that the proposed interpolant cannot be written in the barycentric form of Berrut and Mittelmann \cite{1997berrut}, where the $x$ dependence occurs only through the factors $(x - x_j)^{-1}$. That said, the proposed interpolant can be made to resemble the barycentric form of Schneider and Werner \cite{1991schneider} by finding the full partial fraction decomposition of every term in Equation \ref{eq:alt_numerator}. This is not difficult to do by means of the residue method, but the result is found to be less numerically stable than Equation \ref{eq:alt_bary_full} and is not discussed further.

One of the outstanding qualities of the Floater--Hormann interpolants is the provable absence of poles on the real line. A slight modification of Floater and Hormann's theorem \cite{2007floater} is enough to show that the proposed interpolants share this property.
\begin{theorem}
For all $0 \leq d \leq n$ and $0 \leq e \leq d$, the rational interpolant $r^{(d, e)}(x)$ in Equation \ref{eq:alt_blend} has no poles in $\R$.
\label{thm:poles}
\end{theorem}
\begin{proof}
Multiply the numerator and denominator of $r^{(d, e)}(x)$ in Equation \ref{eq:alt_blend} by $\ell(x)$ in Equation \ref{eq:alt_all_factors}, and let $s^{(d, e)}(x)$ be the denominator. It is sufficient to show that $s^{(d, e)}(x) > 0$ for all $x \in \R$. To that end, define additional nodes $x_{i} = x_{0}$ for $i \in \set{-e, -e + 1, \dots , -1}$ and $x_{i} = x_{n}$ for $i \in \set{n, n + 1, \dots , n + e}$, define the index set $I := \set{-e, -e + 1, \dots , n + e}$, and define the functions
\begin{equation*}
\mu_{i}(x) = \prod_{\mathclap{j = -e}}^{i - 1} (x - x_j) \prod_{\mathclap{k = i + d + 1}}^{n + e} (x_k - x).
\end{equation*}
Then $s^{(d, e)}(x)$ can be written as
\begin{equation*}
s^{(d, e)}(x) = \sum_{i \in I} \mu_{i}(x).
\end{equation*}
At this point, the proof that $s^{(d, e)}(x) > 0$ for $x \in \R \backslash \set{x_0, x_n}$ is identical to that of Floater and Hormann \cite{2007floater} up to a relabeling of the indices. Their proof does not extend to $x \in \set{x_0, x_n}$ because of the assumption that all of the $x_i$ are distinct.

First consider $x = x_0$. Since $e \leq d$, the first factor in $\mu_{-e}(x)$ is $(x_j - x)$ for some $j > 0$, and $\mu_{-e}(x_0) > 0$. For all other $i > -e$, the first factor in $\mu_{i}(x)$ is $(x - x_{-e})$ and $\mu_{i}(x_0) = 0$. Hence
\begin{equation*}
s^{(d, e)}(x_0) = \sum_{i \in I} \mu_{i}(x_0) = \mu_{-e}(x_0) > 0.
\end{equation*}
The reasoning to show that $s^{(d, e)}(x_n) > 0$ involves the last factors in the $\mu_{i}(x)$ but is otherwise the same. \qed
\end{proof}

\begin{corollary}
For all $0 \leq d \leq n$ and $0 \leq e \leq d$, the rational interpolant $r^{(d, e)}(x)$ has no unattainable points.
\end{corollary}
\begin{proof}
For any $\alpha \in \set{0, 1, \dots , n}$, define the index set $J_\alpha := \set{i \in I: \mu_i(x_\alpha) \neq 0}$. Theorem \ref{thm:poles} requires that $J_\alpha$ be nonempty, since otherwise $\sum_{i \in I} \mu_{i}(x_\alpha) = 0$ and $r^{(d, e)}(x)$ would have a pole at $x_\alpha$. Let $q_i(x)$ be the $i$th local polynomial interpolant in Equation \ref{eq:alt_blend}, and observe from the definition of $r^{(d, e)}(x)$ and the fact that polynomial interpolants do not have unattainable points that $q_i(x_\alpha) = y_\alpha$ whenever $\mu_i(x_\alpha) \neq 0$. Then
\begin{equation*}
r^{(d, e)}(x_\alpha) = \frac{\sum_{i \in J_\alpha} \mu_i(x_\alpha) q_i(x_\alpha)}{\sum_{i \in J_\alpha}^{\vphantom{a}} \mu_i(x_\alpha)} = y_\alpha \frac{\sum_{i \in J_\alpha} \mu_i(x_\alpha)}{\sum_{i \in J_\alpha}^{\vphantom{a}} \mu_i(x_\alpha)} = y_\alpha
\end{equation*}
and $r^{(d, e)}(x)$ interpolates the data at $x_\alpha$. \qed
\end{proof}

\begin{figure}
\center
\subfloat[]{%
	\label{subfig:m_blending}{%
		\includegraphics[height=37mm]{%
			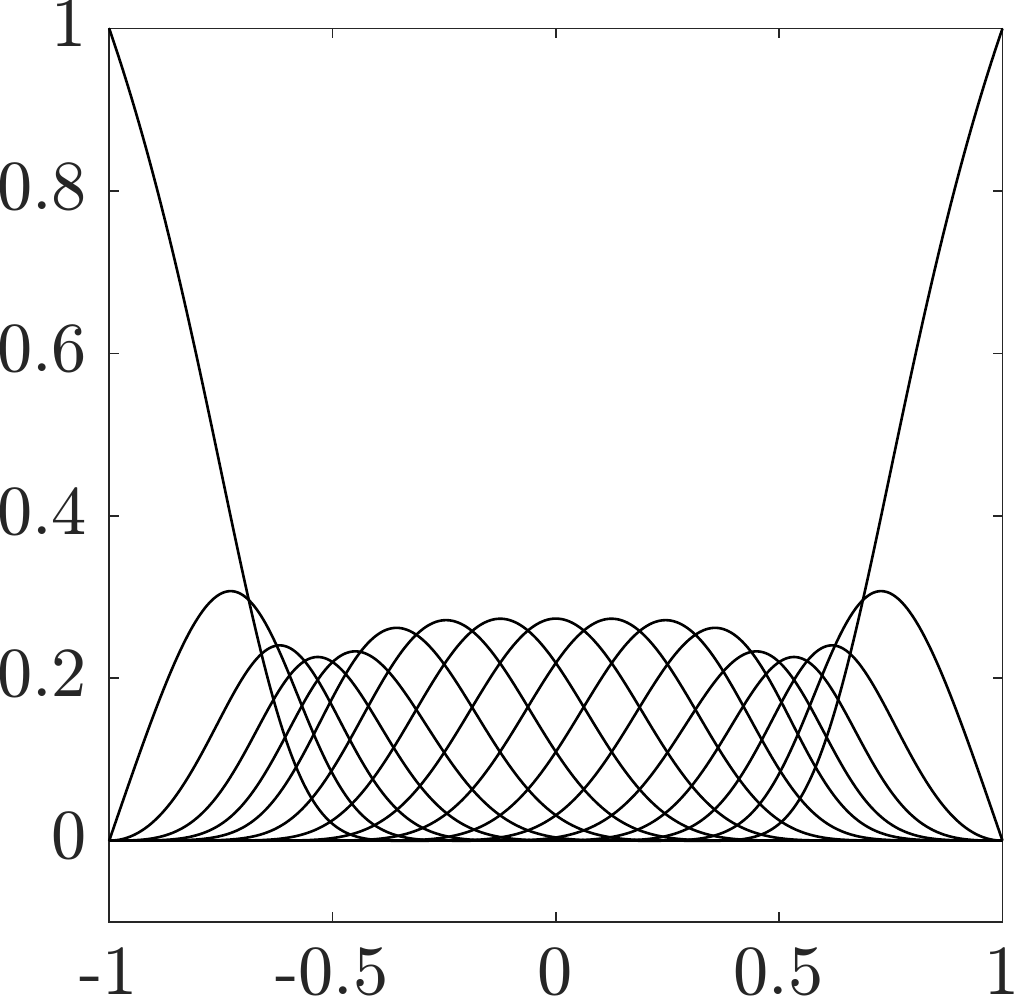}}} \quad
\subfloat[]{%
	\label{subfig:m_basis}{%
		\includegraphics[height=37mm]{%
			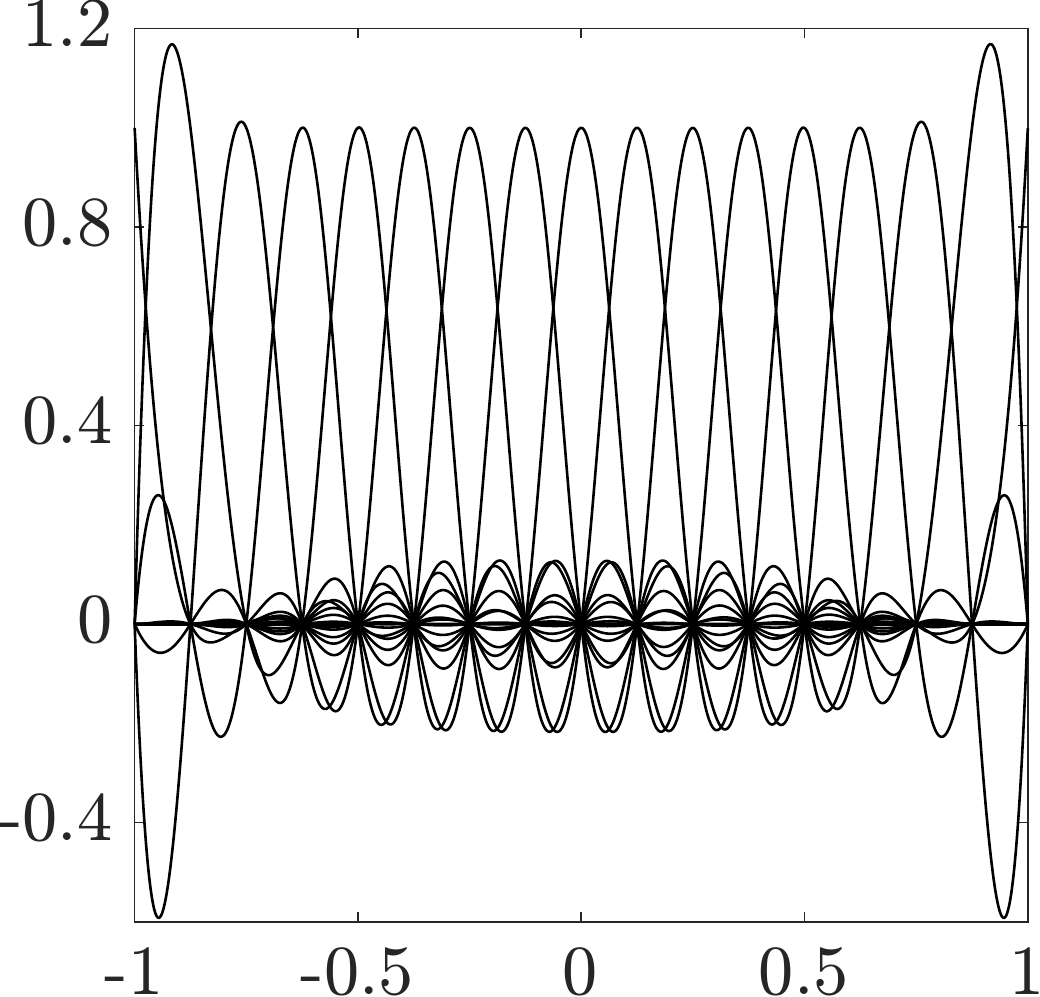}}} \quad
\subfloat[]{%
	\label{subfig:m_lebesgue}{%
		\includegraphics[height=37mm]{%
			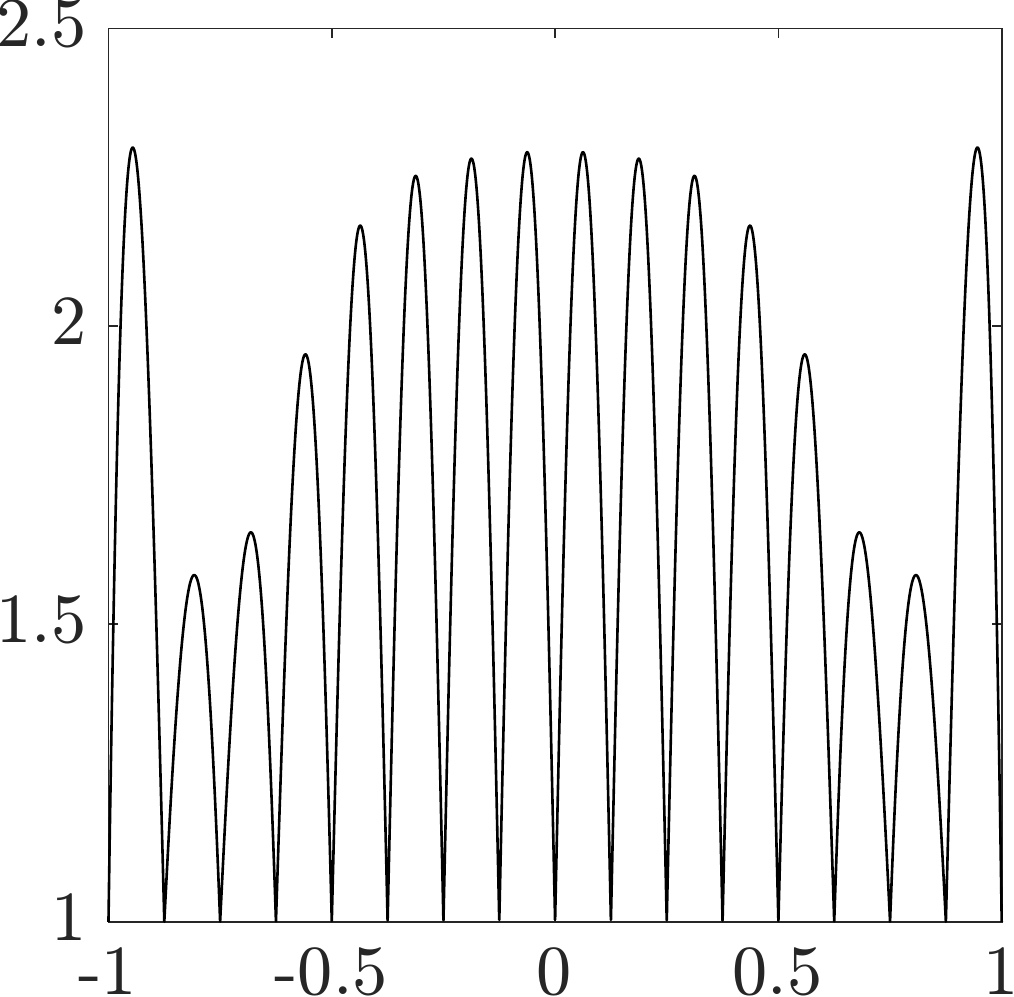}}}
\caption{\label{fig:m_example}For the proposed interpolant with $d, e = 8, 4$ on $17$ equispaced nodes in the interval $[-1, 1]$, (a) the blending functions $\mu_i(x) / \sum_{j \in I} \mu_j(x)$ in Equation \ref{eq:alt_blend}, (b) the basis functions $\beta^{(8, 4)}_j(x)$ in Equation \ref{eq:alt_basis}, and (c) the Lebesgue function $\lambda^{(8, 4)}(x)$.}
\end{figure}

Finally, $r^{(d, e)}(x)$ reproduces polynomials of degree at most $d - e$. If $f$ is such a polynomial, then $q_i(x) = f(x)$ for all $i \in I$ and
\begin{equation*}
r^{(d, e)}(x) = \frac{\sum_{i \in I} \mu_i(x) q_i(x)}{\sum_{i \in I}^{\vphantom{a}} \mu_i(x)} = f(x) \frac{\sum_{i \in I} \mu_i(x)}{\sum_{i \in I}^{\vphantom{a}} \mu_i(x)} = f(x).
\end{equation*}
This suggests that the proposed interpolant be compared with a Floater--Hormann interpolant of degree $d - e$. The counterpart to Figure \ref{fig:fh_example} would then be, e.g., Figure \ref{fig:m_example} where the behavior of the proposed interpolant with $d = 8$ and $e = 4$ is shown. Comparing the ends of the interval in the two figures reveals an increase in the number of blending functions $\mu_i(x) / \sum_{j \in I} \mu_j(x)$ in Figure \ref{subfig:m_blending}, the damping of the oscillations of the basis functions $\beta^{(8, 4)}_j(x)$ in Figure \ref{subfig:m_basis}, and a reduction in the Lebesgue function $\lambda^{(8, 4)}(x) = \sum_i |\beta^{(8, 4)}_{i}(x)|$ in Figure \ref{subfig:m_lebesgue}. That is, the proposed modification of the Floater--Hormann interpolants had the intended effect of improving the conditioning at the ends of the interval. Moreover, if the interpolation error is $O(h^{\delta + 1})$ where $h$ is the node spacing and $\delta$ is the degree of the local polynomial interpolants, then the proposed interpolant should have lower interpolation error on the interior of the interval than the corresponding Floater--Hormann interpolant. This is supported by the numerical results in the following section.

\section{Numerical results}
\label{sec:numeric}

The Floater--Hormann interpolants have the property that once the weights are known, the computational complexity to find the value of the interpolant at some $x$ is $O(n)$ \cite{2004berrut}. Ideally, any modification of the Floater--Hormann interpolants would have the same property. This is the case for $r^{(d, e)}(x)$, for which the additional computational complexity over a Floater--Hormann interpolant is $O(d e)$. This can be seen from the $O(d)$ values of $\zeta_{j}^{(d, e)}(x)$ and $\eta_{j}^{(d, e)}(x)$ in Equation \ref{eq:alt_bary_full}, each of which requires $O(e)$ operations to calculate using Equations \ref{eq:zeta_horner} and \ref{eq:eta_horner}. Practically, while the additional computational complexity is small only for small values of $e$, this is sufficient for the cases of interest.

This section specifically considers the behavior of $r^{(d, e)}(x)$ as an approximant for three variations of Runge's function $f(x) = 1 / (1 + x^2)$ on equispaced nodes. The first uses the interval $[-5, 5]$ since this appears often in the literature \cite{2007floater,2013klein}. The second breaks the bilateral symmetry and uses the interval $[-3, 7]$ to better represent general $f$. The third uses the interval $[-5, 5]$, but perturbs the $y_i$ with Gaussian noise with $\sigma = 10^{-8}$ to simulate the measurement error when $r^{(d, e)}(x)$ is used as an interpolant. These examples are not comprehensive and the results of this section certainly do not carry the weight of proof, but they do seem to be representative of the performance of $r^{(d, e)}(x)$ in practice.

\begin{figure}
\center
\subfloat[]{%
	\label{subfig:lebesgue_de}{%
		\includegraphics[height=39mm]{%
			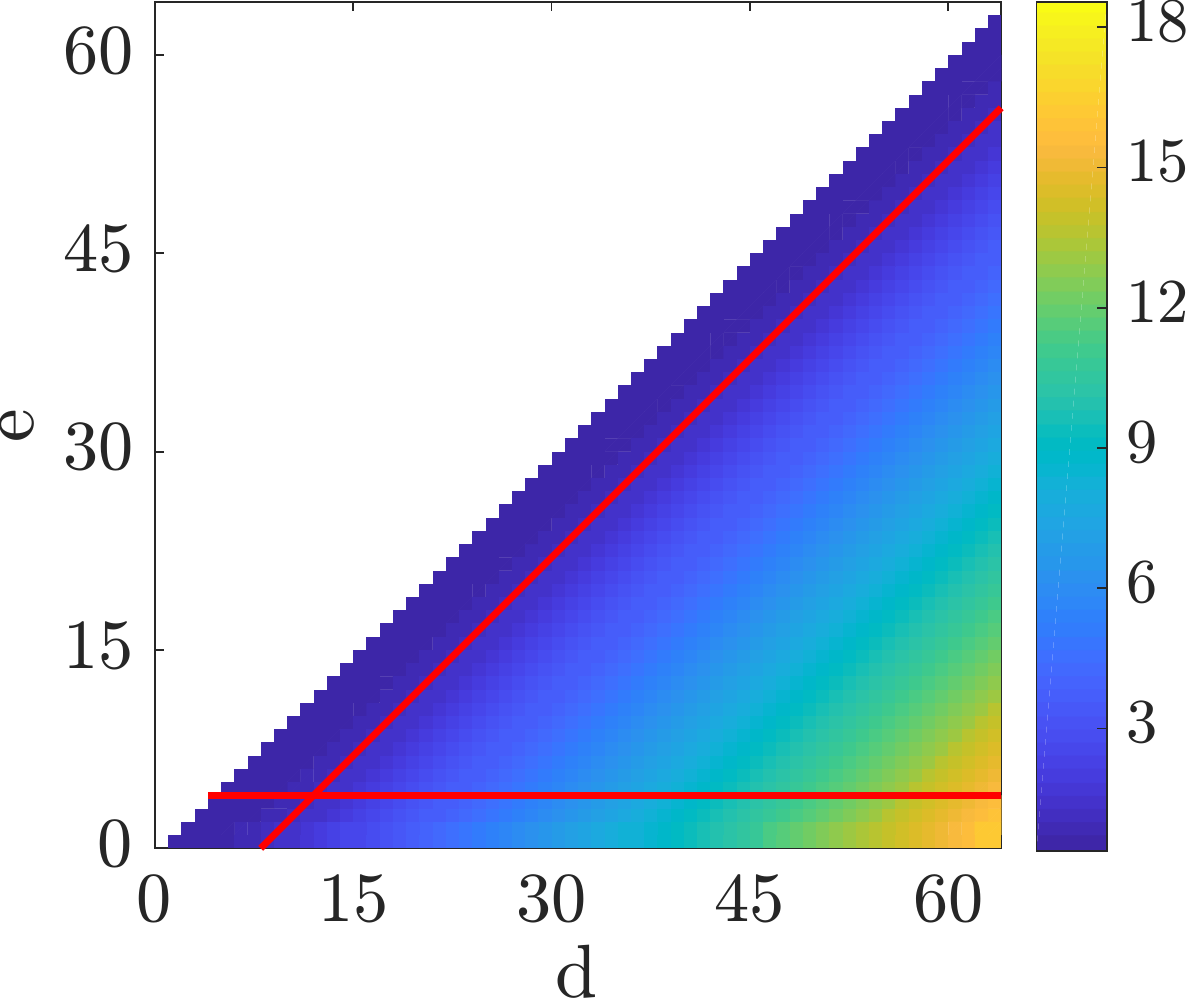}}} \quad
\subfloat[]{%
	\label{subfig:runge_s_de}{%
		\includegraphics[height=39mm]{%
			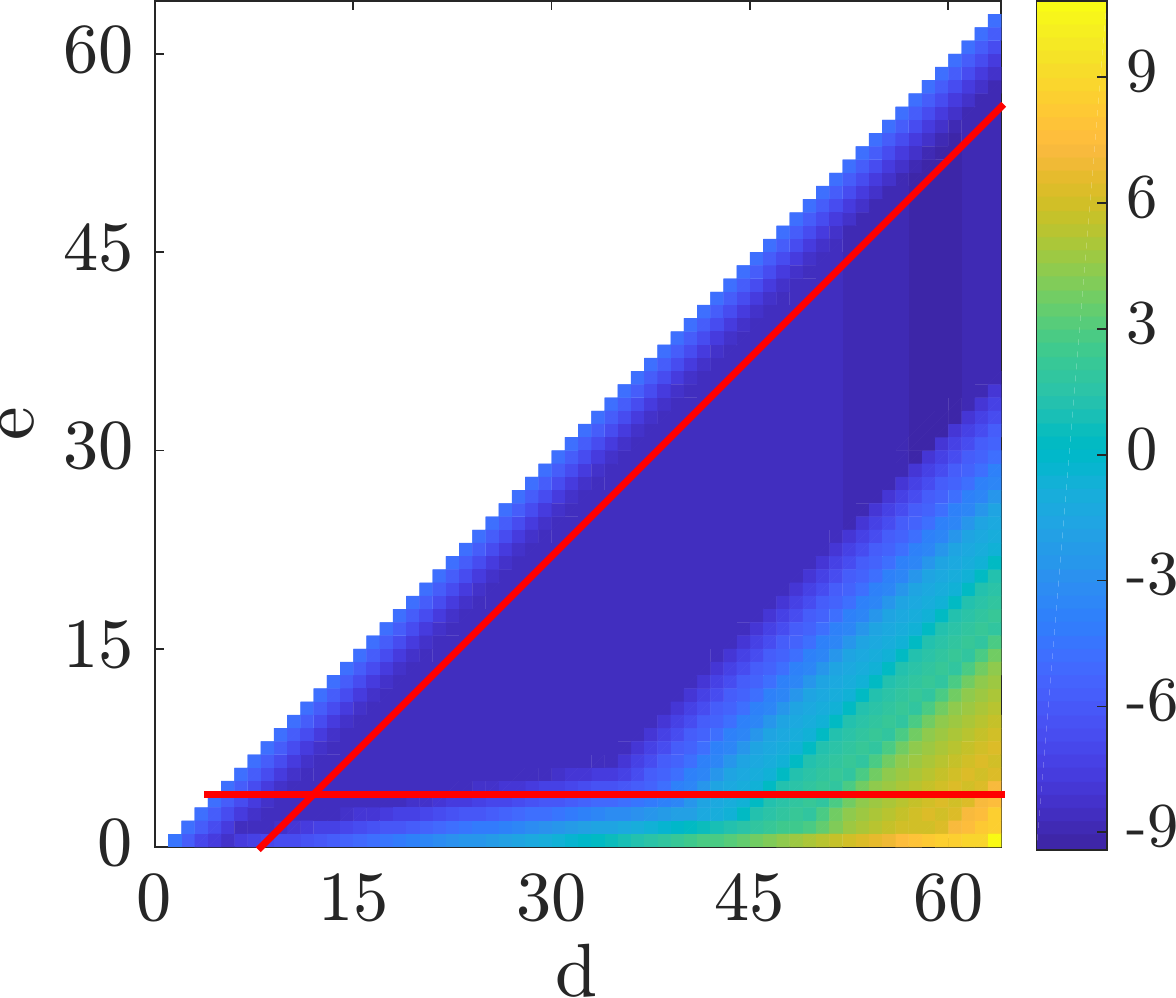}}} \quad
\subfloat[]{%
	\label{subfig:runge_a_de}{%
		\includegraphics[height=39mm]{%
			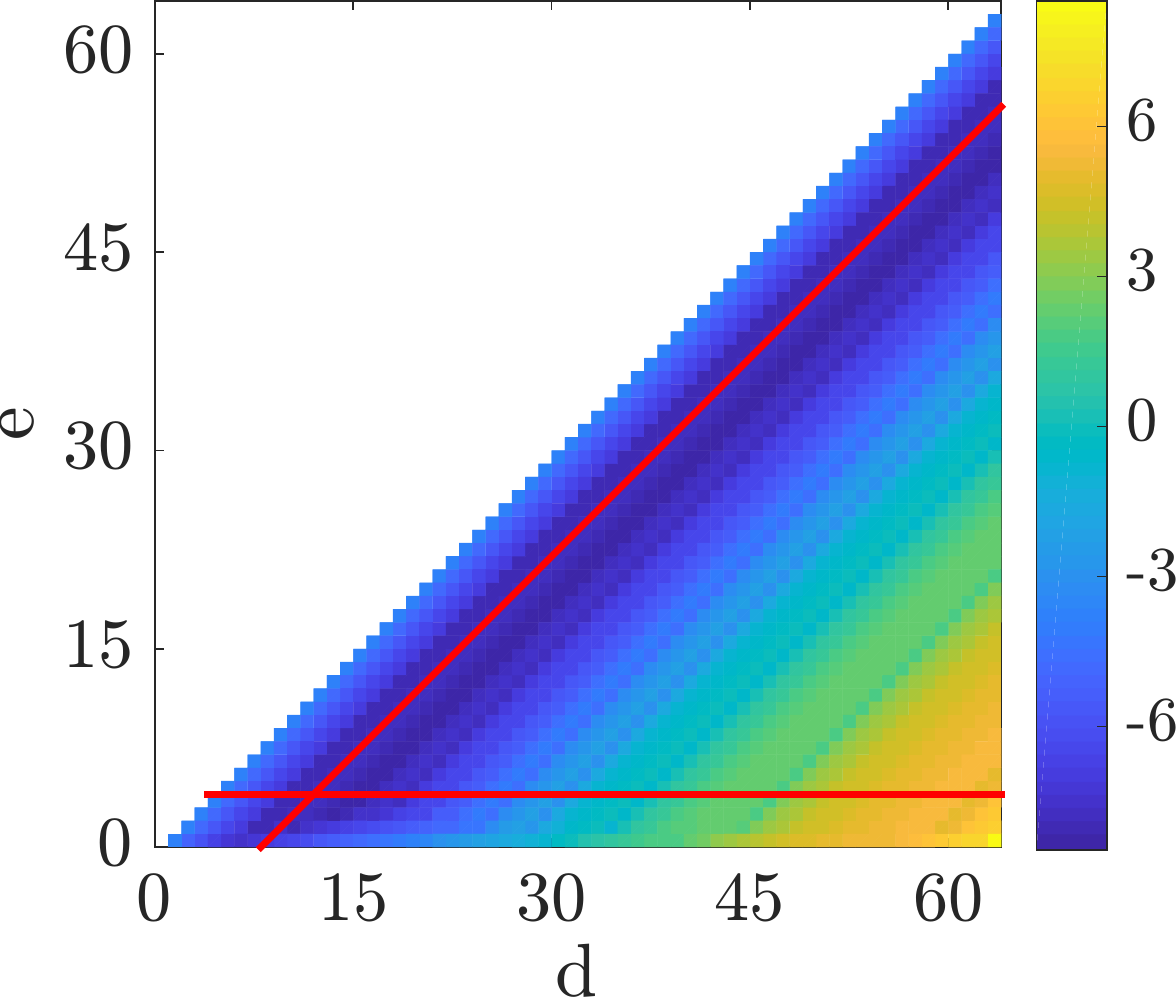}}} \quad
\subfloat[]{%
	\label{subfig:runge_sn_de}{%
		\includegraphics[height=39mm]{%
			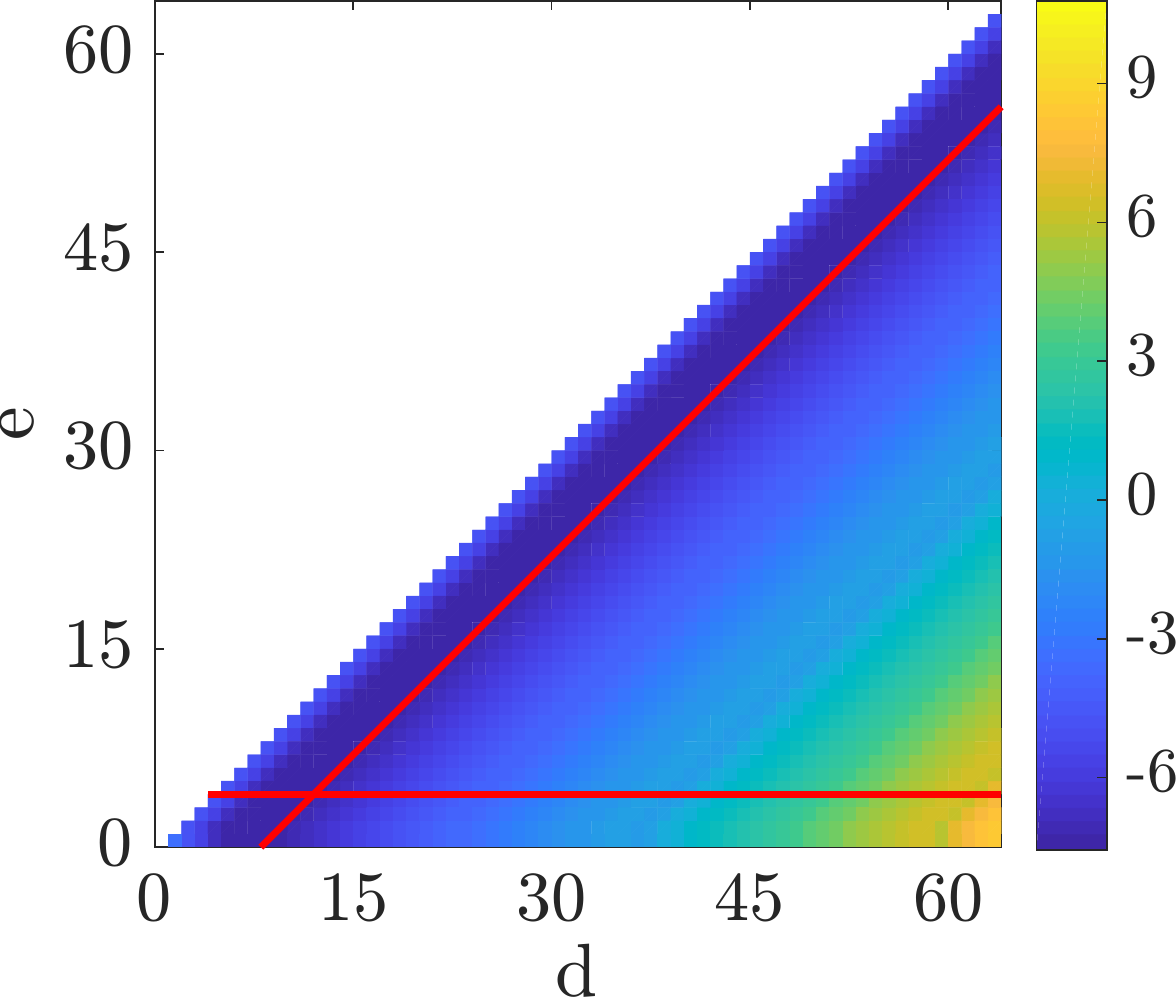}}}
\caption{\label{fig:de_scan}For $r^{(d, e)}(x)$ and $n = 64$ with equispaced nodes, (a) the Lebesgue constant, (b) $L^\infty$ error of $f(x) = 1 / (1 + x^2)$ on the interval $[-5, 5]$, (c) $L^\infty$ error of $f(x) = 1 / (1 + x^2)$ on the interval $[-3, 7]$, (d) $L^\infty$ error of $f(x) = 1 / (1 + x^2)$ on the interval $[-5, 5]$ with Guassian noise with $\sigma = 10^{-8}$. The base $10$ logarithm of the values is reported. The solid red lines in each figure are given by $e = 4$ and $e = d - 8$.}
\end{figure}

Figure \ref{fig:de_scan} considers the behavior of $r^{(d, e)}(x)$ for $n = 64$ as a function of $d$ and $e$. The Lebesgue constant is shown in Figure \ref{subfig:lebesgue_de}, and increases exponentially with $d$ for any fixed $e$ (this has been proven for $e = 0$ \cite{2012bos}) with the exception of the region $e \geq d - 5$ where the Lebesgue constant is small and nearly constant. This supports the supposition that the behavior of the Lebesgue constant is dominated by the polynomial interpolants of degree at most $d - e$ at the ends of the interval. The $L^{\infty}$ error for Runge's function on the interval $[-5, 5]$ is shown in Figure \ref{subfig:runge_s_de}, and is small and nearly constant in the region defined by $e \leq d - 5$, $e \geq (d - 5) / 5$ and $e \geq d - 28$. While this intersects $e = 0$ at a single point, increasing the value of $e$ dramatically expands the useful interval of $d$ and helps to stabilize the behavior of the approximant. The $L^{\infty}$ error for Runge's function on the interval $[-3, 7]$ is shown in Figure \ref{subfig:runge_a_de}, and is small in the region defined by $e \leq d - 5$, $e \geq (d - 5) / 5$ and $e \geq d - 15$. This better represents general $f$, and while the useful interval of $d$ is smaller the behavior is essentially the same as for Figure \ref{subfig:runge_s_de}. Finally, the $L^{\infty}$ error for Runge's function on the interval $[-5, 5]$ with Gaussian noise is shown in Figure \ref{subfig:runge_sn_de}, and is small and nearly constant in the region defined by $e \leq d - 4$ and $e \geq d - 8$. The resemblance to the Lebesgue constant in Figure \ref{subfig:lebesgue_de} is to be expected, since the Lebegsue constant effectively indicates the sensitivity of the approximant to measurement errors.

One advantage of the Chebyshev and spline interpolants is that they have few adjustable parameters---the absence of trade-offs in the parameter values makes using them a straightforward affair. For the Floater--Hormann interpolants, there is the temptation to increase $d$ to reduce the interpolation error, but this carries the risk of amplifying the rounding error. While there has been some progress on finding the optimal $d$ when approximating analytic functions \cite{2012guttel}, this requires knowledge of the closest singularity in the complex plane. The proposed interpolant apparently complicates the situation further by introducing a second adjustable parameter $e$. Observe though that the line $e = d - 8$ intersects the regions of small $L^{\infty}$ error in Figures \ref{subfig:runge_s_de}, \ref{subfig:runge_a_de} and \ref{subfig:runge_sn_de}, and that most of stability gains with increasing $e$ have already been achieved on the line $e = 4$. These intersect at the point $d, e = 12, 4$, which seems to offer a reasonable balance of accuracy, stability and computational expense for general use. When $r^{(d, e)}(x)$ is used as an approximant and the measurement error is negligible, increasing $d$ to $12 \leq d \leq 16$ can improve the approximation rate at the cost of some stability.

\begin{figure}
\center
\subfloat[]{%
	\label{subfig:runge_s_lf}{%
		\includegraphics[height=47mm]{%
			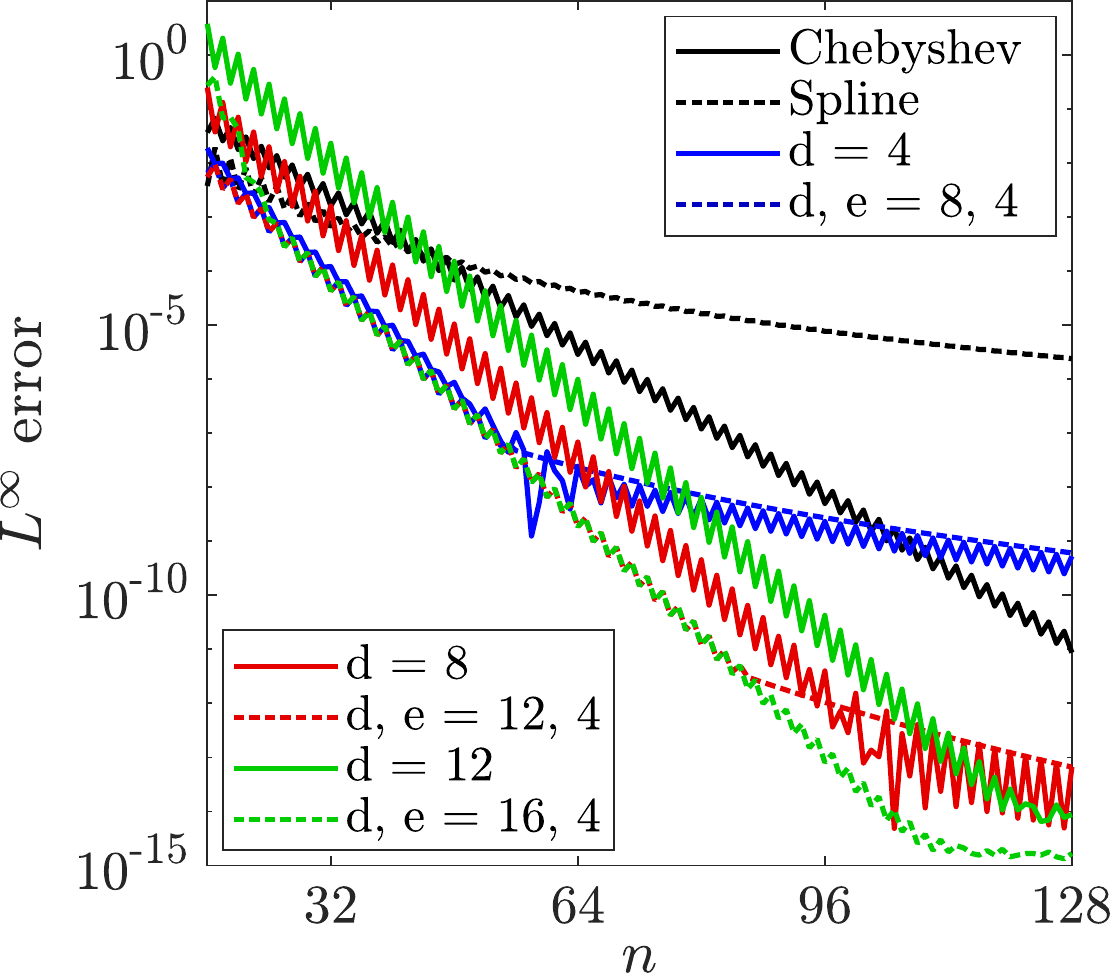}}} \quad
\subfloat[]{%
	\label{subfig:runge_s_l1}{%
		\includegraphics[height=47mm]{%
			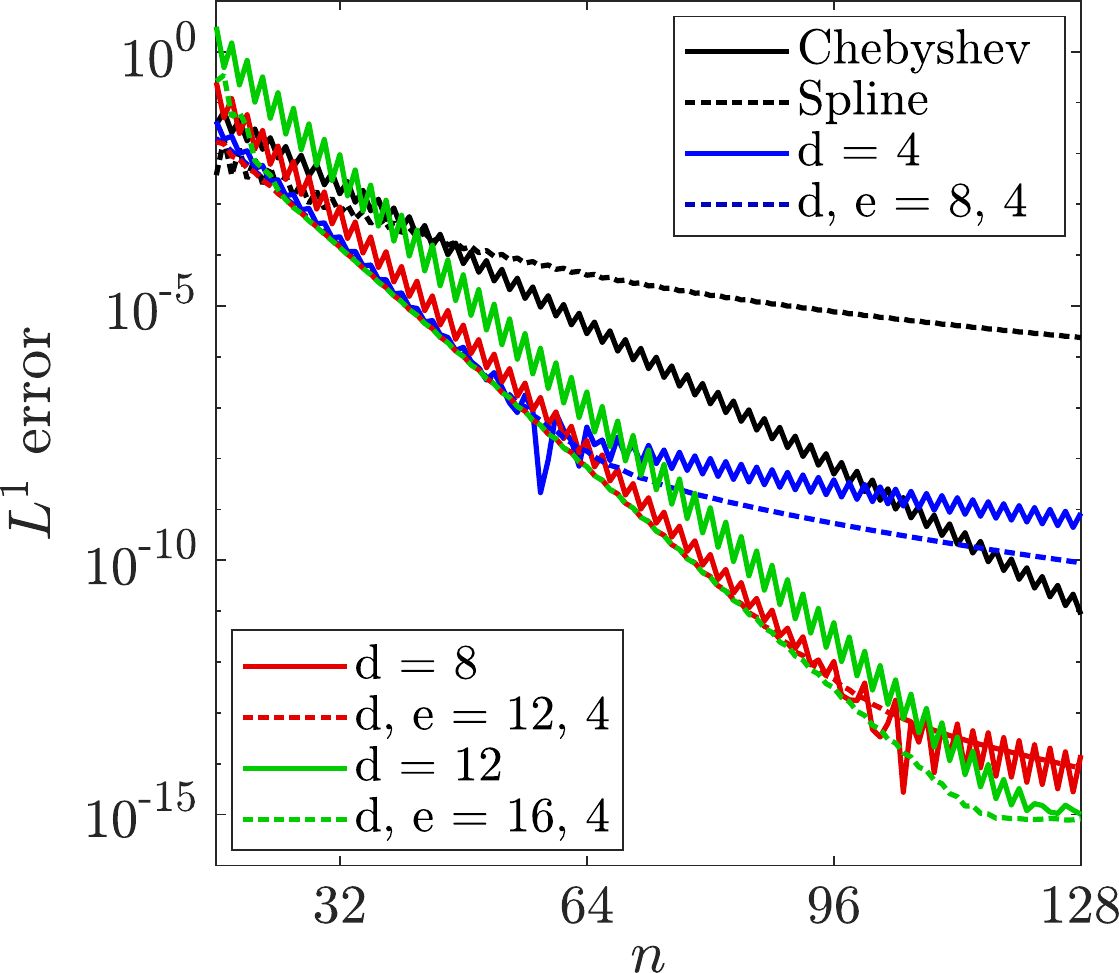}}} \quad
\subfloat[]{%
	\label{subfig:runge_a_lf}{%
		\includegraphics[height=47mm]{%
			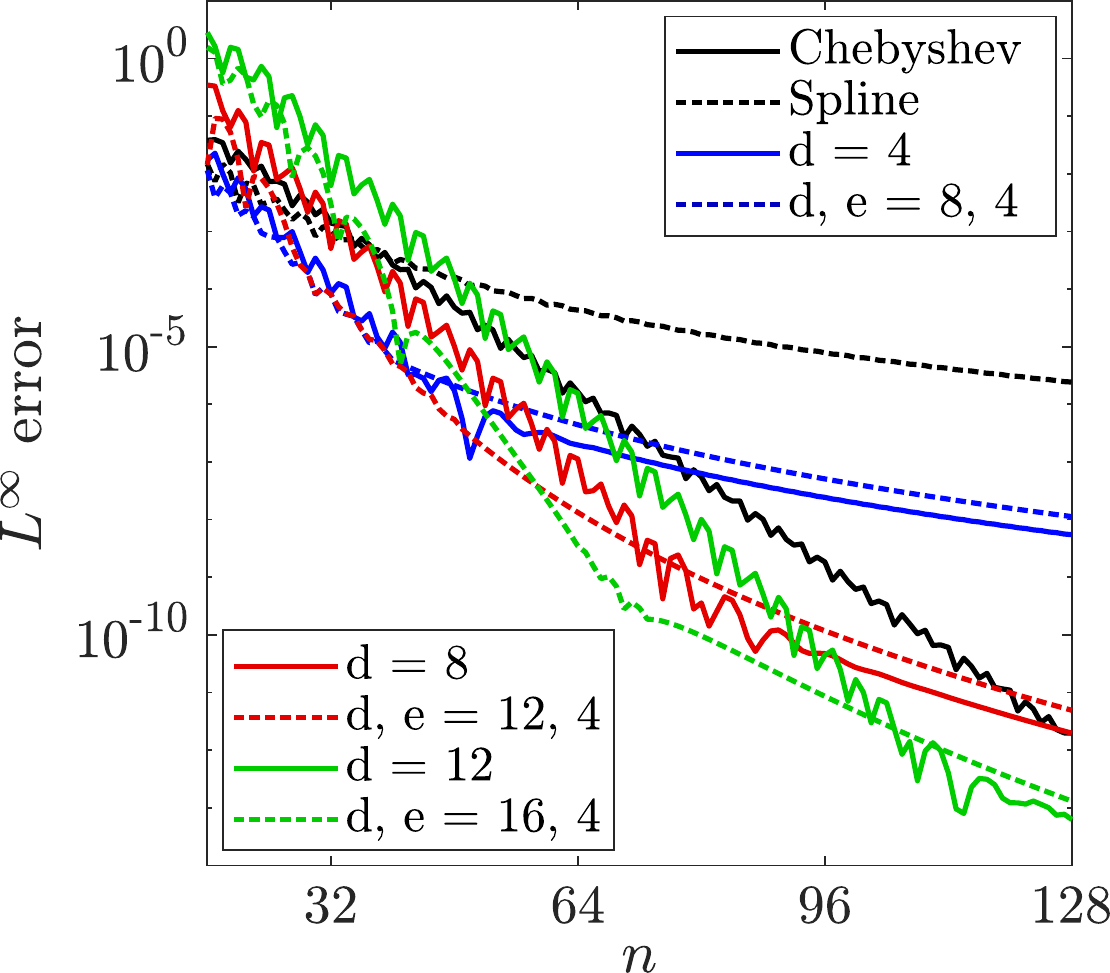}}} \quad
\subfloat[]{%
	\label{subfig:runge_a_l1}{%
		\includegraphics[height=47mm]{%
			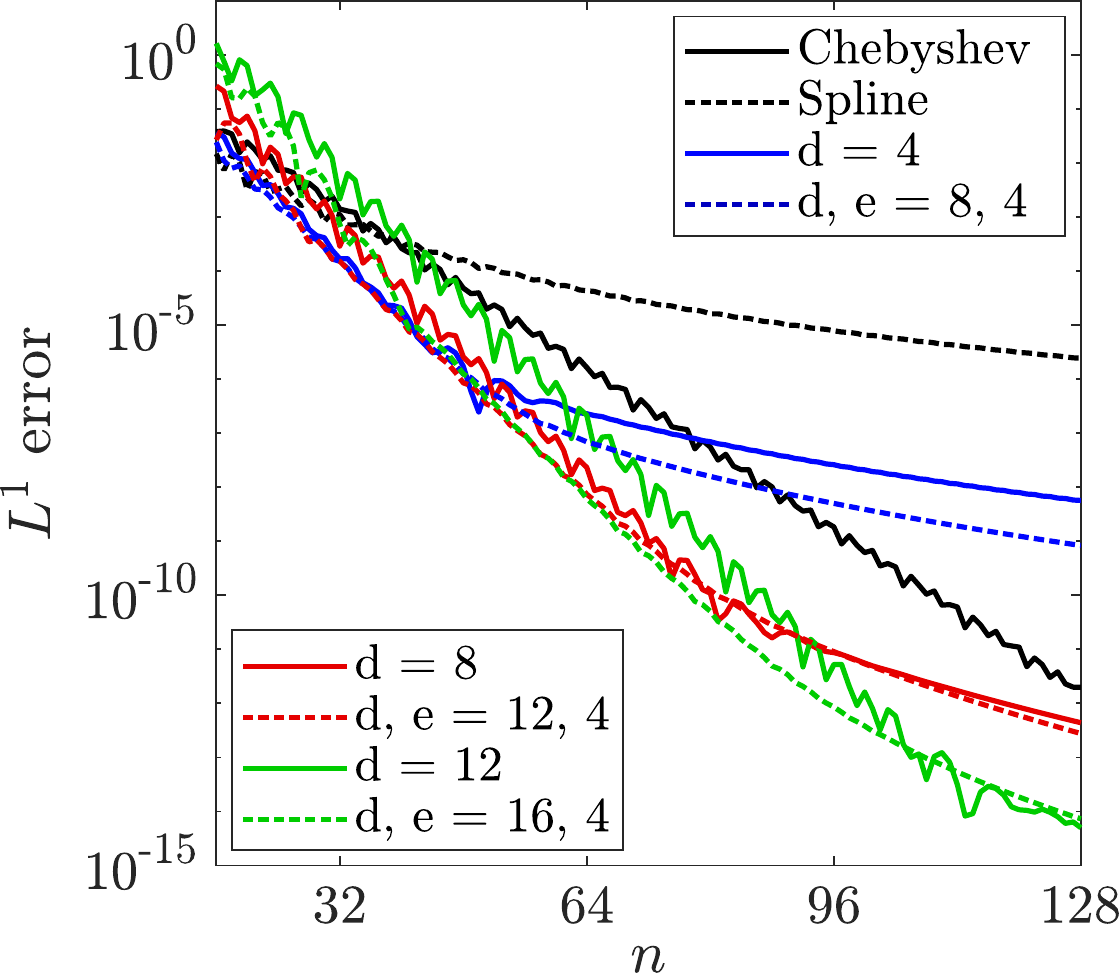}}} \quad
\subfloat[]{%
	\label{subfig:runge_sn_lf}{%
		\includegraphics[height=47mm]{%
			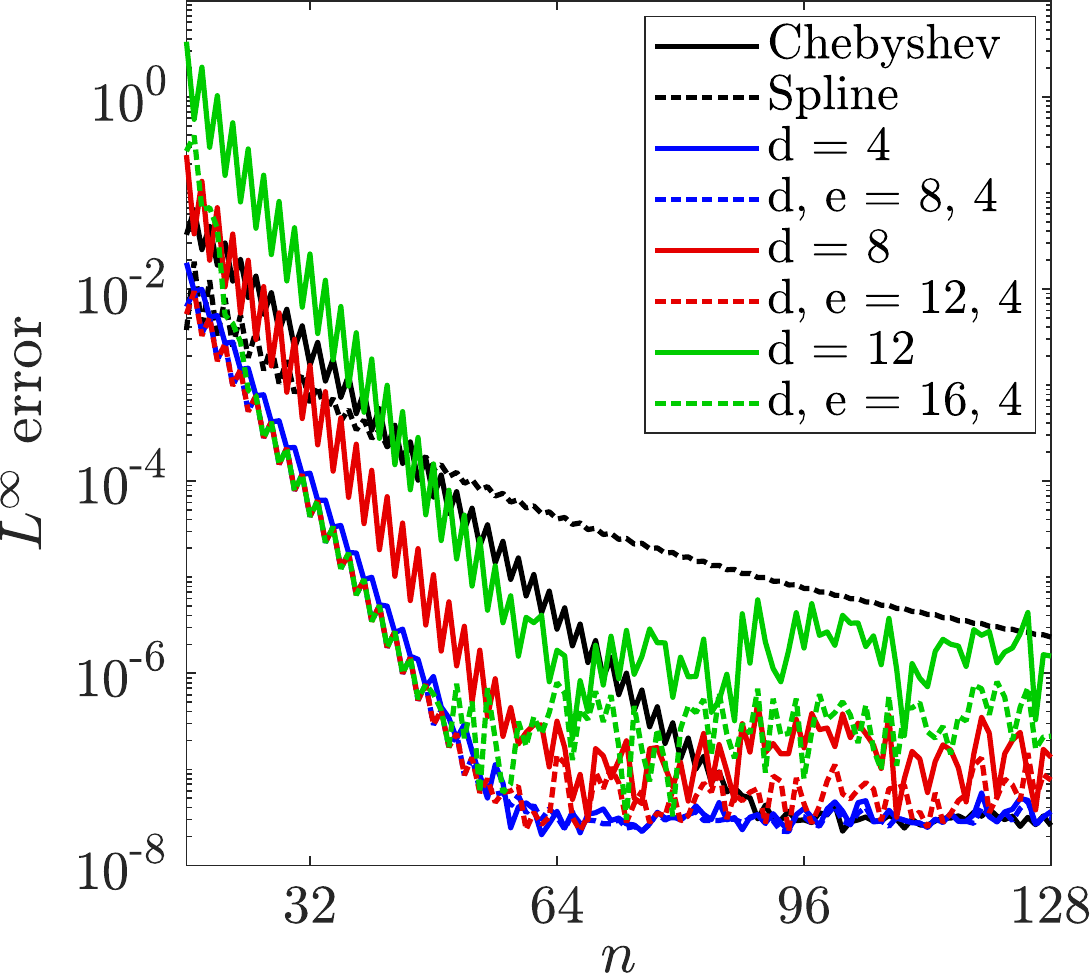}}} \quad
\subfloat[]{%
	\label{subfig:runge_sn_l1}{%
		\includegraphics[height=47mm]{%
			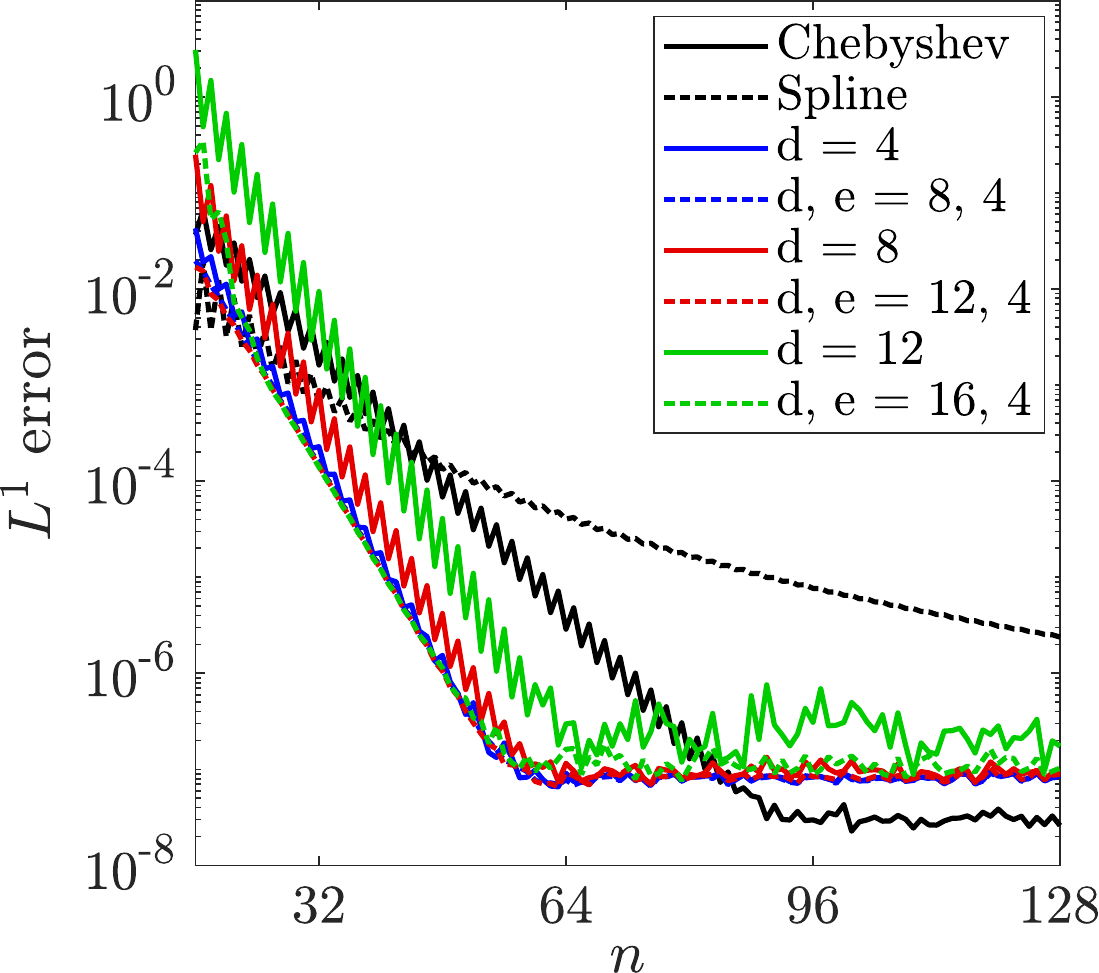}}}
\caption{\label{fig:comparison}$L^{\infty}$ and $L^{1}$ errors of eight approximants as functions of $n$ for $f(x) = 1 / (1 + x^2)$. The solid black line is for a polynomial through Chebyshev points of the second kind, the dashed black line is for a cubic spline, the remaining solid lines are for $r^{(d)}(x)$, and the remaining dashed lines are for $r^{(d, e)}(x)$. (a) and (b) use the interval $[-5, 5]$, (c) and (d) use the interval $[-3, 7]$, and (e) and (f) use the interval $[-5, 5]$ with Guassian noise with $\sigma = 10^{-8}$.}
\end{figure}

The performance of the proposed approximant as a function of $n$ is compared with that of Chebyshev, cubic spline, and Floater--Hormann approximants for Runge's function in Figure \ref{fig:comparison}. Figures \ref{subfig:runge_s_lf} and \ref{subfig:runge_s_l1} use the interval $[-5, 5]$, Figures \ref{subfig:runge_a_lf} and \ref{subfig:runge_a_l1} use the interval $[-3, 7]$, and Figures \ref{subfig:runge_sn_lf} and \ref{subfig:runge_sn_l1} use the interval $[-5, 5]$ with Gaussian noise. There are a number of observations to be made from these figures. First, the Chebyshev and cubic spline approximants show the expected exponential and power law convergence. The Floater--Hormann and proposed approximants initially exhibit exponential convergence and afterward power law convergence, with the transition occurring at an $n$ that increases with $d$. Although Platte, Trefethen and Kuijlaars \cite{2011platte} have proven that there is no approximant on equispaced nodes that is stable and converges exponentially in the limit $n \rightarrow \infty$, this is not as serious an issue as one might believe---the Floater--Hormann and proposed approximants have errors that are often substantially less than and reach the level of machine precision well before the Chebyshev approximant. One conclusion then is that there can be some situations where the Floater--Hormann and proposed approximants on equispaced nodes are preferable to the Chebyshev approximant, despite the proven properties \cite{2013trefethen} of the latter.

Second, the proposed approximant $r^{(d, e)}(x)$ seems more stable than the Floater--Hormann approximant $r^{(d - e)}(x)$ in three respects. First, the approximantion rate of a Floater--Hormann approximant depends on the parity of $n$ \cite{2007floater}, as is visible from the oscillations of the $L^{\infty}$ and $L^{1}$ errors in Figures \ref{subfig:runge_s_lf} and \ref{subfig:runge_s_l1}. The same oscillations are strongly damped for the $L^{\infty}$ error and almost completely absent for the $L^{1}$ error of the proposed approximant. Second, the $L^{\infty}$ error of the Floater--Hormann approximants in the interval of exponential convergence increases exponentially with $d$ for $d > 4$, as is visible from the vertical offset of the curves in Figures \ref{subfig:runge_s_lf}, \ref{subfig:runge_a_lf} and \ref{subfig:runge_sn_lf}. The $L^{\infty}$ error of the proposed approximants in the same interval instead collapses onto a single curve in Figures \ref{subfig:runge_s_lf} and \ref{subfig:runge_sn_lf}, with the exception of $n < 24$ for $d, e = 16, 4$. Third, the amplification of the random errors in Figures \ref{subfig:runge_sn_lf} and \ref{subfig:runge_sn_l1} is smaller for $r^{(d, e)}(x)$ than for $r^{(d - e)}(x)$, sometimes by an order of magnitude. The reason for this is not clear, since Figure \ref{subfig:lebesgue_de} shows that $r^{(d, e)}(x)$ and $r^{(d - e)}(x)$ have similar Lebesgue constants.

Third, the proposed approximants appear to have an advantage over the Floater--Hormann approximants with regard to error. The $L^{\infty}$ and $L^{1}$ errors of $r^{(d, e)}(x)$ are often just above and just below the respective errors of $r^{(d - e)}(x)$ in the interval of power law convergence, as is visible in Figures \ref{subfig:runge_s_lf}, \ref{subfig:runge_s_l1}, \ref{subfig:runge_a_lf} and \ref{subfig:runge_a_l1}. Since the $L^{\infty}$ error bounds the pointwise error from above, the Floater--Hormann approximants are preferable here. That said, the proposed approximants can have smaller approximation errors by several orders of magnitude in the interval of exponential convergence, and apparently amplify the random errors substantially less as well. This arguably makes the proposed approximants preferable for the case of general $f$ and $n$.

\begin{table*}
\center
\caption{\label{tbl:runge_compare}$L^{\infty}$ and $L^{1}$ errors of $r^{(d)}(x)$ and $r^{(d, e)}(x)$ on the interval $[-5, 5]$ with $n$ equispaced nodes for $f(x) = 1 / (1 + x^2)$. For $r^{(d)}(x)$, $d$ is the optimal value reported by Floater and Hormann \cite{2007floater}. For $r^{(d, e)}(x)$, $d, e = 14, 4$ except when $d$ is constrained by $n$.}
\begin{tabular}{r r l l r l l}
\hline
\hline \noalign{\smallskip}
 \multicolumn{1}{r}{$n$} &
 \multicolumn{1}{r}{$d$} &
 \multicolumn{1}{c}{$L^\infty$, $r^{(d)}(x)$} &
 \multicolumn{1}{c}{$L^1$, $r^{(d)}(x)$} &
 \multicolumn{1}{r}{$d, e$} &
 \multicolumn{1}{c}{$L^\infty$, $r^{(d, e)}(x)$} &
 \multicolumn{1}{c}{$L^1$, $r^{(d, e)}(x)$} \\
\hline \noalign{\smallskip}
  10 & 0 & $3.606 \times 10^{-2}$ & $1.601 \times 10^{-1}$ & 10, 4 & $3.005 \times 10^{-2}$ & $1.243 \times 10^{-1}$ \\
  20 & 1 & $1.536 \times 10^{-3}$ & $6.656 \times 10^{-3}$ & 14, 4 & $1.674 \times 10^{-3}$ & $4.519 \times 10^{-3}$ \\
  40 & 3 & $4.307 \times 10^{-6}$ & $1.306 \times 10^{-5}$ & 14, 4 & $3.463 \times 10^{-6}$ & $1.220 \times 10^{-5}$ \\
  80 & 7 & $2.038 \times 10^{-10}$ & $8.003 \times 10^{-11}$ & 14, 4 & $1.214 \times 10^{-11}$ & $4.684 \times 10^{-11}$ \\
  160 & 10 & $1.887 \times 10^{-15}$ & $9.230 \times 10^{-16}$ & 14, 4 & $1.887 \times 10^{-15}$ & $9.226 \times 10^{-16}$ \\
\hline
\hline
\end{tabular}
\end{table*}

From Figure \ref{fig:comparison}, the question of the optimal $d$ for a Floater--Hormann approximant with a fixed $n$ seems to involve (at least for Runge's function) finding the smallest $d$ where the interval of exponential convergence includes $n$---further increasing $d$ can displace the entire curve vertically and increase the approximation error. The advantage of the proposed approximant then is that for modest $e$, the value of $d$ can be safely increased and the interval of exponential convergence expanded without adversely affecting the approximation error. While this sometimes results in smaller overall errors, the more significant practical advantage is that this reduces the necessity of adjusting $d$ to find the optimal value. For example, Table \ref{tbl:runge_compare} reproduces and expands upon a table in Floater and Hormann \cite{2007floater} that reports the optimal values of $d$ for Runge's function with $n$ equispaced nodes on the interval $[-5, 5]$. Since $r^{(d)}(x)$ and $r^{(d, e)}(x)$ are used as approximants and the measurement error is minimal, $d, e = \min(14, n), 4$ is used for the proposed approximant instead of the more conservative $d, e = 12, 4$ suggested above. Observe that even without adjusting $d$ and $e$, the $L^{1}$ error of the proposed approximant is less than that of the optimal Floater--Hormann approximant for every $n$ in the table, and the $L^{\infty}$ error is less than or equal to that of the optimal Floater--Hormann approximant for every $n$ with the exception of $n = 20$. For $n = 80$, the $L^{\infty}$ error of the proposed approximant is more than an order of magnitude smaller. At the very least then, the proposed approximants could be useful when finding the optimal value of $d$ for a Floater--Hormann approximant would be difficult or time consuming.

\section{Conclusion}
\label{sec:conclusion}

If one is presented with data on equispaced nodes and desires to interpolate further values, the Floater--Hormann family of rational interpolants is a good choice. They are infintely smooth, more stable than the polynomial interpolant of minimum degree, and often more accurate than cubic spline interpolants. That said, the Floater--Hormann interpolants are a family, and one immediately encounters the question of which one to use in practice. There are at least three answers available in the literature. First, Floater and Hormann \cite{2007floater} seem to suggest using a small value of $d$ (perhaps $d = 3$) and increasing $n$ until the desired accuracy is achieved. This is a conservative approach, and appears to be widely used \cite{2007press}. Second, G\"uttel and Klein \cite{2012guttel} describe a procedure whereby $d$ is given as a function of $n$ depending on the analyticity of $f$. This approach is elegant, but is also more complicated and relies on knowledge of $f$ that is not always available. Third, Klein \cite{2013klein} proposed extrapolating to points outside of the interval and constructing Floater--Hormann interpolants on the extended set of points. Certain published results suggest that these Extended Floater--Hormann interpolants do not suffer the usual drawbacks from high values of $d$, but others \cite{2017decamargo} suggest that the extrapolation process is a source of significant instability. Klein explicitly ignored this source of error in his analysis.

The rational interpolants proposed in Equations \ref{eq:alt_blend} and \ref{eq:alt_bary_full} are modifications of the Floater--Hormann interpolants that blend additional local polynomial interpolants at the ends of the interval. This two-parameter family initially appears to make the above question more difficult to answer, but various numerical examples suggest that a narrow interval of parameter values offers a good balance of accuracy, stability and computational expense. While the accuracy is comparable to that of the Floater--Hormann interpolant with optimal $d$, the proposed interpolants achieve this for constant $d$ and $e$. This is envisioned as simplifying the use of the rational interpolants in practical contexts, where the user might not have enough knowledge of the function $f$ to use a more sophisticated alternative.


%

\end{document}